\documentclass[11pt]{amsart}
\usepackage{amsmath,amsfonts,amsthm,amssymb,url,graphicx,algorithm,algorithmicx,tikz,float,hyperref}

\usepackage[all]{xy}
\usepackage{footmisc}
\usepackage{amsmath}
\usepackage{bigints}
\usepackage{fontawesome}
\usepackage{tikz-cd}
\usepackage{mathrsfs}
\usepackage{amssymb}
\usepackage{amsbsy}
\usepackage{dsfont}
\usepackage{amsthm}
\usepackage{graphicx}
\usepackage{xypic}
\usepackage{CJKutf8}
\usepackage{mathtools}
\usepackage{nccmath}
\usepackage{skak}
\usepackage{enumerate}
\usepackage{commath} 
\usepackage{tikz}
\usepackage{relsize}
\usepackage{faktor}
\usepackage{pst-plot}
\usepackage{float}
\usepackage[shortlabels]{enumitem}
\usetikzlibrary{decorations.markings}
\usetikzlibrary{arrows}
\usetikzlibrary{arrows.meta}
\usepackage{etoolbox}

\usepackage[backend=bibtex,style=ieee,firstinits=true, doi=false,isbn=false,url=false,eprint=false, sorting=anyt]{biblatex}
\DeclareFieldFormat[article]{volume}{\mkbibbold{#1}}
\DeclareFieldFormat*{title}{\textit{#1}}

\usepackage{listings}

\usepackage{lipsum}

\hypersetup{
           breaklinks=true,   
           colorlinks=true,   
           pdfusetitle=true,  
              linkcolor=black,citecolor=black             
        }

\makeatletter
\newcommand{\pushright}[1]{\ifmeasuring@#1\else\omit$\displaystyle#1$\ignorespaces\fi}
\makeatother
\makeatletter
\def\moverlay{\mathpalette\mov@rlay}
\def\mov@rlay#1#2{\leavevmode\vtop{%
   \baselineskip\z@skip \lineskiplimit-\maxdimen
   \ialign{\hfil$\m@th#1##$\hfil\cr#2\crcr}}}
\newcommand{\charfusion}[3][\mathord]{
    #1{\ifx#1\mathop\vphantom{#2}\fi
        \mathpalette\mov@rlay{#2\cr#3}
      }
    \ifx#1\mathop\expandafter\displaylimits\fi}
\makeatother

\definecolor{indigo}{RGB}{75, 0, 130}
\definecolor{midnightblue}{RGB}{25, 25, 112}

\makeatletter

\AtBeginDocument{%
 \@ifpackageloaded{refcheck}{%
 \@ifundefined{hyperref}{}{%
 \let\T@ref@orig\T@ref
 \def\T@ref#1{\T@ref@orig{#1}\wrtusdrf{#1}}%
 \let\@refstar@orig\@refstar
 \def\@refstar#1{\@refstar@orig{#1}\wrtusdrf{#1}}%
 \DeclareRobustCommand\ref{\@ifstar\@refstar\T@ref}%
 }%
 }{}%
}
\makeatother

\usepackage{pgfplots}

\pgfplotsset{compat=1.6}

\pgfplotsset{soldot/.style={color=blue,only marks,mark=*}} \pgfplotsset{holdot/.style={color=blue,fill=white,only marks,mark=*}}

\def\XXint#1#2#3{{\setbox0=\hbox{$#1{#2#3}{\int}$ }
\vcenter{\hbox{$#2#3$ }}\kern-.6\wd0}}

\usetikzlibrary{arrows,shapes,positioning}
\usetikzlibrary{decorations.markings}
\tikzstyle arrowstyle=[scale=1]
\tikzstyle directed=[postaction={decorate,decoration={markings,
    mark=at position .65 with {\arrow[arrowstyle]{stealth}}}}]
\tikzstyle reverse directed=[postaction={decorate,decoration={markings,
    mark=at position .65 with {\arrowreversed[arrowstyle]{stealth};}}}]

\DeclareSymbolFont{extraup}{U}{zavm}{m}{n}
\DeclareMathSymbol{\varheart}{\mathalpha}{extraup}{86}
\DeclareMathSymbol{\vardiamond}{\mathalpha}{extraup}{87}

\usepackage{color}
\usepackage{mathrsfs}
\usepackage{yfonts}
\usepackage{multirow}
\usepackage{bbm}

\usepackage[toc]{glossaries}
\usepackage{glossary-mcols}
\usepackage{chngcntr}
\usepackage{apptools}
\AtAppendix{\counterwithin{lem}{chapter}}
\AtAppendix{\counterwithin{prop}{chapter}}
\AtAppendix{\counterwithin{theo}{chapter}}
\AtAppendix{\counterwithin{definition}{chapter}}
\AtAppendix{\counterwithin{cortheo}{chapter}}
\AtAppendix{\counterwithin{corprop}{chapter}}

\makeglossaries

\usepackage[utf8]{inputenc}
\usepackage[english]{babel}

\numberwithin{equation}{section}

\newtheorem{lem}{Lemma}
\newtheorem{theo}{Theorem}

\numberwithin{theo}{section}
\numberwithin{lem}{section}
\numberwithin{prop}{section}

\numberwithin{corlem}{section}
\numberwithin{cortheo}{section}
\numberwithin{corprop}{section}

\theoremstyle{definition}

\newtheorem{conjecture}{Conjecture}
\newtheorem*{remark}{Remark}

\numberwithin{definition}{section}
\numberwithin{exmp}{section}
\numberwithin{exer}{section}
\numberwithin{conjecture}{section}

\makeatletter
\newcommand*{\defeq}{\mathrel{\rlap{%
                     \raisebox{0.24ex}{$\m@th\cdot$}}%
                     \raisebox{-0.24ex}{$\m@th\cdot$}}%
                     =}
\makeatother

\newcommand{\lt}{\ensuremath <}
\newcommand{\gt}{\ensuremath >}

\renewcommand{\le}{\leqslant}
\renewcommand{\ge}{\geqslant}

\newcommand{\ve}{\varepsilon}

\newcommand{\Ga}{\Gamma}
\newcommand{\al}{\alpha}

\newcommand{\be}{\beta}
\newcommand{\de}{\delta}

\newcommand{\ze}{\zeta}
\newcommand{\f}{\frac}
\newcommand{\mf}{\mfrac}
\newcommand{\tf}{\tfrac}

\newcommand{\cs}{\mathscr}

\newcommand{\bb}{\mathbb}
\newcommand{\dt}{\dif t}

\renewcommand{\mod}[1]{\ (\mathrm{mod}\ #1)}

\DeclarePairedDelimiterX{\inn}[2]{\langle}{\rangle}{#1, #2}

\DeclareMathOperator{\me}{e}

\DeclareMathOperator{\sgn}{sgn}

\DeclareSymbolFont{eulargesymbols}{U}{zeuex}{m}{n}
\DeclareMathSymbol{\intop}{\mathop}{eulargesymbols}{"52}

\title[The first negative Fourier coefficient]{The first negative Fourier coefficient of an Eisenstein series newform}

\author{Sebastián Carrillo Santana}
\address{Mathematics Institute, Utrecht University,  Hans Freudenthalgebouw, Budapestlaan 6, 3584 CD Utrecht, Netherlands}
\email{s.carrillosantana@uu.nl}
\begin{document}
\begin{abstract}
There have been a number of papers on statistical questions concerning the sign changes of Fourier coefficients of newforms. In one such paper, Linowitz and Thompson gave a conjecture describing when, on average, the first negative sign of the Fourier coefficients of an Eisenstein series newform occurs. In this paper, we correct their conjecture and prove the corrected version.
\end{abstract}

\maketitle


\section{Introduction}
For a Dirichlet character $\chi$ and a positive integer $N$, we will denote by $ M_k(N,\chi)$ the vector space of modular forms on $\Ga_0(N)$ of weight $k$, level $N$ and character $\chi$. Let $E_k(N,\chi)$ be the subspace of Eisenstein series and $S_k(N,\chi)$ the subspace of cusp forms. For a prime $p$, we let $T_p$ be the $p$th Hecke operator.

Let $H_k^{*}(N)$ be the subspace of $S_k(\chi_0,N)$ of newforms with trivial character $\chi_0$. Given a newform $f\in H_k^{*}(N)$, let $\lambda_f(p)$ be the eigenvalue of $f$ with respect to the Hecke operator $T_p$. The restriction to the trivial character ensures that the sequence $\{\lambda_f(p)\}$ is real. Many authors have studied the sequence of signs of the Hecke eigenvalues of $f$. For example, one could pose questions such as:
\begin{enumerate}[{\rm (i)}]
  \item Are there infinitely many primes $p$ such that $\lambda_f(p)\gt 0$ (or $\lambda_f(p)\lt 0$)?
  \item What is the first change of sign? More specifically, what is the smallest $n\ge 1$ (or prime $p$) such that $\lambda_f(n)\lt 0$ (or $\lambda_f(p)\lt 0$)? This is an analogue of the least quadratic non-residue problem.

  \item Given an arbitrary sequence of signs $\ve_p\in\{\pm 1\}$, what is the number of newforms $f$ (in some family) such that $\sgn\lambda_f(p)=\ve_p$ for all $p\le x$?
\end{enumerate}
In the cusp form setting, questions (i) and (ii) are answered in \cite{Kohnen}, \cite{Kowalski}, and \cite{Matomaki}. In this paper, we focus on (iii). Kowalski, Lau, Soundararajan and Wu \cite{Kowalski} obtained a lower bound for the proportion of newforms $f\in H_k^{*}(N)$ whose sequence of eigenvalues $\lambda_f(p)$ has signs coinciding with a prescribed sequence $\{\ve_p\}$:
\begin{theo}[Kowalski, Lau, Soundararajan, Wu, 2010]
  Let $N$ be a squarefree number, $k\ge 2$ an even integer, and $\{\ve_p\}$ a sequence of signs. Then, for any $0\lt \ve \lt \mf12$, there exists some $c\gt 0$ such that
  \[
  \frac{1}{|H_k^{*}(N)|}\# \{f\in H_k^{*}(N) \; : \; \sgn\lambda_f(p)=\ve_p \mbox{ for } p\le z,\; p\nmid N\} \ge \Big(\f12 - \ve\Big)^{\pi(z)}
  \]
  for $z=c\sqrt{\log{kN}\log\log{kN}}$ provided $kN$ is large enough. Here $\pi(z)$ is the number of primes less than or equal to $z$.
\end{theo}
Now, let $\chi_1$, $\chi_2$ be Dirichlet characters modulo $N_1, N_2$ and for an integer $k\gt 2$, define the following variant of the sum of divisors function:
\begin{equation}\label{eq: sigma(n)}
  \sigma_{\chi_1,\chi_2}^{k-1}(n)=\sum_{d|n} \chi_1\big(\f{n}{d}\big)\chi_2(d)d^{k-1}.
\end{equation}
Now assume that $\chi_1$ and $\chi_2$ are not simultaneously principal $\mod{1}$. It is well known (see, for example, \cite{Diamond}) that  if $\chi_1\chi_2(-1)=(-1)^k$, then the function
\[
E_k(\chi_1,\chi_2,z)\defeq \frac{\de(\chi_1)}{2}L(1-k,\chi_2) + \sum_{n\ge 1}\sigma_{\chi_1,\chi_2}^{k-1}(n)q^n,
\]
is an Eisenstein series of weight $k$, level $N_1N_2$ and character $\chi_1\chi_2$. Here $q=\me^{2\pi i z}$ and
\[
\de(\chi_1)=\begin{cases}
              1, & \mbox{if } \chi_1 \mbox{ is principal} \\
              0, & \mbox{otherwise}.
            \end{cases}
\]
In 1977, Weisinger \cite{Weisinger} developed a newform theory for $E_k(N,\chi)$ analogous to the one developed by Atkin and Lehner \cite{Atkin} for cusp forms. In this theory, we have:
\begin{itemize}
  \item The newforms of $E_k(N,\chi)$ are functions of the form $E_k(\chi_1,\chi_2,z)$ for which $N=N_1N_2$, $\chi=\chi_1\chi_2$, and $\chi_1,\chi_2$ are primitive.
  \item The eigenvalue of $E_k(\chi_1,\chi_2,z)$ with respect to the Hecke operator $T_p$ is $\sigma_{\chi_1,\chi_2}^{k-1}(p)$. In other words, the eigenvalues of this type of Eisenstein series coincide with its Fourier coefficients.
\end{itemize}
By exploiting the analytical properties of $\sigma_{\chi_1,\chi_2}^{k-1}(n)$, Linowitz and Thompson \cite{Lola} answered the three questions mentioned at the beginning of this article for Eisenstein series newforms.

Note that by \eqref{eq: sigma(n)}, $\sigma_{\chi_1,\chi_2}^{k-1}(n)\in\bb{R}$ when $\chi_1,\chi_2$ are real characters. Since we want $E_k(\chi_1,\chi_2,z)$ to be an Eisenstein series, we exclude the case when $\chi_1$ and $\chi_2$ are principal. We call these types of characters \emph{quadratic} because for every fundamental discriminant $D$, i.e., for each discriminant arising from a quadratic number field, we can associate a real character defined by $\chi_D(m)=\big(\frac{D}{m}\big)$. Therefore, counting Eisenstein series newforms of level $N\le x$ is equivalent to counting fundamental discriminants $D_1, D_2$ with $|D_1D_2|\le x$. Let
\[
\cs{D}\defeq \{(D_1, D_2) \; : \; |D_1D_2|\le x\}.
\]
Taking all of these facts into consideration, Linowitz and Thompson \cite{Lola} showed:
\begin{theo}[Linowitz, Thompson, 2015]\label{th: LolaMainTheorem}
  Let $\{p_1,\ldots, p_k\}$ be a sequence of primes and $\{\ve_{p_1},\ldots, \ve_{p_k}\}\in\{-1,0,1\}$ a sequence of signs. Then,
  \begin{align*}
  \frac{1}{|\cs{D}|}\#\{(D_1,D_2)\in\cs{D} \; :\; &{}\sgn\sigma_{\chi_1,\chi_2}^{k-1}(p_i)=\ve_{p_i},\; 1\le i \le k\} \\ &{}\xrightarrow[x\to \infty]{}\prod_{\substack{\ve_{p_i}=0 \\ 1\le i \le k}}\frac{1}{(p_i+1)^2}\prod_{\substack{\ve_{p_i}\neq 0 \\ 1\le i \le k}}\frac{p_i(p_i+2)}{2(p_i+1)^2}\cdot
  \end{align*}
\end{theo}
Now, let $\eta(D_1,D_2)$ represent the smallest prime $p$ such that $\sgn(\sigma_{\chi_1,\chi_2}^{k-1}(p))=-1$. Linowitz and Thompson \cite{Lola} then conjectured:
\begin{conjecture}\label{con: ConLola}
We have
\[
\frac{\sum_{|D_1D_2|\le x}\eta(D_1,D_2)}{\sum_{|D_1D_2|\le x} 1} \xrightarrow[x\to \infty]{} \theta,
\]
where
\begin{equation}\label{eq: theta}
\theta\defeq \sum_{k=1}^{\infty}\frac{p_k^2(p_k+2)}{2(p_k+1)^2}\prod_{j=1}^{k-1}\frac{2+p_j(p_j+2)}{2(p_j+1)^2}\approx 3.9750223902\ldots
\end{equation}
\end{conjecture}
They gave a heuristic argument as evidence towards their conjecture, showing:
\begin{align*}
  \frac{\sum_{|D_1D_2|\le x}\eta(D_1,D_2)}{\sum_{|D_1D_2|\le x} 1}  \xrightarrow[x\to \infty]{}{}&\sum_{k=1}^{\infty}p_k\, \mbox{Prob}(\eta(D_1,D_2)=p_k) \\
  ={} & \sum_{k=1}^{\infty}p_k\, \mbox{Prob}(\ve_{p_k}=-1)\prod_{i=1}^{k-1}\mbox{Prob}(\ve_{p_i}=0\mbox{ or }1) \\
  ={}& \sum_{k=1}^{\infty}\frac{p_k^2(p_k+2)}{2(p_k+1)^2}\prod_{i=1}^{k-1}\bigg(\frac{1}{(p_i+1)^2}+\frac{p_i(p_i+2)}{2(p_i+1)^2}\bigg),
\end{align*}
where the last equality follows from Theorem \ref{th: LolaMainTheorem}. The problem with this argument is that Theorem \ref{th: LolaMainTheorem} fixes a set of primes and then lets $x\to\infty$. In this argument we need to allow the primes to tend to infinity with $x$. The authors stated: ``[W]e have a good understanding of the effect of the small primes, but one would need to argue that the primes after some cutoff point do not make much of an impact on the average. Presumably, this would require using the large sieve''.

The goal of the present article is to correct their conjecture by proving the following result:
  \begin{theo}\label{th: Main Theorem}
  We have
  \[
\frac{\sum_{|D_1D_2|\le x}\eta(D_1,D_2)}{\sum_{|D_1D_2|\le x} 1} \xrightarrow[x\to \infty]{} {}\Theta\cdot(1-\be)+\al,
\]
where
\[
\Theta =\sum_{k=1}^{\infty}\frac{p_k^2}{2(p_k+1)}\prod_{j=1}^{k-1}\frac{p_j+2}{2(p_j+1)},
\]

\[
  \al= \sum_{k=1}^{\infty}\frac{p_k^2}{2(p_k+1)^2}\prod_{j=1}^{k-1}\frac{p_j+2}{2(p_j+1)},
\]
and
\[
  \be= \sum_{k=1}^{\infty}\frac{p_k}{2(p_k+1)^2}\prod_{j=1}^{k-1}\frac{p_j+2}{2(p_j+1)}\cdot
\]
Numerically,
\[
\Theta\cdot(1-\be)+\al\approx 4.63255603509332\ldots
\]
\end{theo}
The numerical computation was done using Sage. We used RIF for interval arithmetic and we truncated at $k=1000$. 

\section{Main Tools}
First we will need asymptotic estimates for some sets of fundamental discriminants. It is well known (see, for example, \cite{Cohen}) that
\begin{equation}\label{eq: D le x}
\sum_{|D|\le x} 1 \sim \frac{x}{\ze(2)},
\end{equation}
where $D$ runs over all fundamental discriminants with $|D|\le x$. Here $\ze$ is the Riemann zeta function. Now, let $n_1(m)$ be the smallest integer $n\ge 1$ relatively prime to $m$ such that the congruence $x^2\equiv n\mod{m}$ has no solutions. Even though Vinogradov's conjecture remains open, it is possible to show that large values of $n_1(p)$ are rare. More specifically, using the large sieve, Linnik \cite{Linnik} showed that for all $\ve\gt 0$, we have
\[
\#\{p\le x \; : \; n_1(p)\gt x^{\ve}\}\ll_{\ve} 1.
\]
Using similar ideas to the ones from Linnik's paper, Erd\H{o}s \cite{Erdos} obtained a result concerning the average of $n_1(p)$ as $p$ varies over prime numbers less than or equal to $x$:
\begin{equation}\label{eq: Average n2(p)}
\frac{1}{\pi(x)}\sum_{p\le x}n_1(p)\xrightarrow[x\to \infty]{}\sum_{k=1}^{\infty}\frac{p_k}{2^k},
\end{equation}
where $p_k$ is the $k$th prime and $\pi(x)$ is the prime counting function. In a similar fashion, Pollack \cite{Pollack} considered a variation of \eqref{eq: Average n2(p)}. We summarize his result in the following theorem:
\begin{theo}[Pollack, 2012]\label{th: Pollack}
  For each fundamental discriminant $D$, let $\chi_D$ be the associated Dirichlet character, i.e., $\chi_D(m)\defeq \big(\frac{D}{m}\big)$. For each character $\chi$, let $n_{\chi}$ denote the least $n$ for which $\chi(n)\notin \{0,1\}$. Finally, let $n(D)\defeq n_{\chi_D}$. Then
  \begin{enumerate}[{\rm (i)}]
    \item Uniformly in $k$ such that the $k$th prime satisfies $p_k\le (\log{x})^{\f13}$, we have
    \begin{equation*}\label{eq: D le x, n(D)=p_k}
      \# \{|D|\le x \; : \; n(D)=p_k\}=\frac{p_k}{2(p_k+1)}\prod_{j=1}^{k-1}\frac{p_j+2}{2(p_j+1)}\frac{x}{\ze(2)}+ O(x^{\f23}).
    \end{equation*}
    \item \begin{equation*}\label{eq: n(D) gt (log{x})^{1/3}}
            \sum_{\substack{|D|\le x \\ n(D)\gt (\log{x})^{\f13}}}n(D)=o(x).
          \end{equation*}
  \end{enumerate}
  Therefore, using \eqref{eq: D le x}, we have
  \begin{equation}\label{eq: Average of n(D)}
    \frac{\sum_{|D|\le x}n(D)}{\sum_{|D|\le x} 1} \xrightarrow[x\to \infty]{} \Theta,
  \end{equation}
  where
  \begin{equation*}\label{eq: ThetaDefinition}
    \Theta\defeq \sum_{k=1}^{\infty}\frac{p_k^2}{2(p_k+1)}\prod_{j=1}^{k-1}\frac{p_j+2}{2(p_j+1)}\approx 4.9809473396\ldots
  \end{equation*}
\end{theo}
We will also need the following lemma from Linowitz and Thompson \cite{Lola}:
\begin{lem}\label{lem: ProportionDiscriminants}
  Let $\bb{P}(\ve, p)$ denote the proportion of fundamental discriminants $D$ with $\big( \tf{D}{p}\big)=\ve$. Then, we have
  \[
  \bb{P}(\ve, p)=\begin{cases}
                   \mfrac{p}{2p+2}, & \mbox{if } \ve\in\{\pm 1\} \\
                   \mfrac{1}{p+1}, & \mbox{if } \ve=0.
                 \end{cases}
  \]
\end{lem}
\section{Proof of Theorem \ref{th: Main Theorem}}

Let $\chi_1, \chi_2$ be Dirichlet characters associated with the fundamental discriminants $D_1$ and $D_2$. For a prime $p$,
\[
 \sigma_{\chi_1,\chi_2}^{k-1}(p)=\sum_{d|p} \chi_1\big(\f{p}{d}\big)\chi_2(d)d^{k-1}=\chi_1(p)+\chi_2(p)p^{k-1},
\]
so that
\begin{equation}\label{eq: sgn(sigma)}
  \sgn{\sigma_{\chi_1,\chi_2}^{k-1}(p)}=\begin{cases}
                                        \chi_1(p), & \mbox{if } p|D_2 \\
                                        \chi_2(p), & \mbox{otherwise}.
                                      \end{cases}
\end{equation}

\begin{proof}[Proof of Theorem \ref{th: Main Theorem}]
  By \eqref{eq: D le x},
  \[
  \sum_{|D_1D_2|\le x} 1=\sum_{|D_1|\le x}\sum_{D_2\le \frac{x}{|D_1|}}1\sim \f{x}{\ze(2)}\sum_{|D_1|\le x}\frac{1}{ |D_1|}\cdot
  \]
  Let $A(x)\defeq \sum_{|D|\le x} 1$ and $f(x)\defeq \mf{1}{x}$. Since $A(x)\sim \mf{x}{\ze(2)}$, then by partial summation
  \begin{align}
 \nonumber \sum_{|D_1|\le x}\frac{1}{|D_1|}={}&A(x)f(x)-A(1)f(1)-\int_{1}^{x}A(t)f'(t)\dt \\
  \nonumber \sim{}& \frac{1}{\ze(2)}-1+\int_{1}^{x}\frac{\dt}{\ze(2)t} \\
  \sim{}& \frac{\log{x}}{\ze(2)}\cdot \label{eq: D le x  1/D}
  \end{align}
  Hence
  \begin{equation}\label{eq: denominator}
  \sum_{|D_1D_2|\le x} 1\sim \frac{x\log{x}}{\ze(2)^2}\cdot
  \end{equation}
  Now let us estimate the numerator. For the sake of simplicity, let $\eta\defeq \eta(D_1,D_2)$. Then,
  \[
  \sum_{|D_1D_2|\le x}\eta=\sum_{\substack{|D_1D_2|\le x \\ \eta|D_2}}\eta+\sum_{\substack{|D_1D_2|\le x \\ \eta\nmid D_2}}\eta.
  \]
  If $\eta | D_2$, then by \eqref{eq: sgn(sigma)}, $\eta$ is the smallest prime $p$ such that $\chi_1(p)\notin \{0,1\}$, and with the notation of Theorem \ref{th: Pollack}, this means that $\eta=n(D_1)$. Similarly, if $\eta\nmid D_2$, then $\eta=n(D_2)$. Therefore,
  \[
  \sum_{|D_1D_2|\le x}\eta=\sum_{\substack{|D_1D_2|\le x \\ \eta|D_2}}n(D_1)+\sum_{\substack{|D_1D_2|\le x \\ \eta\nmid D_2}}n(D_2).
  \]
  Now,
  \[
  \sum_{\substack{|D_1D_2|\le x \\ \eta\nmid D_2}}n(D_2)=\sum_{|D_1D_2|\le x }n(D_2)-\sum_{\substack{|D_1D_2|\le x \\ \eta| D_2}}n(D_2),
  \]
  so that
  \begin{equation}\label{eq: Numerator}
    \sum_{|D_1D_2|\le x}\eta=\sum_{|D_1D_2|\le x }n(D_2)+\sum_{\substack{|D_1D_2|\le x \\ \eta|D_2}}n(D_1)-\sum_{\substack{|D_1D_2|\le x \\ \eta|D_2}}n(D_2).
  \end{equation}
  By \eqref{eq: Average of n(D)}, we have
  \begin{align}
  \nonumber \sum_{|D_1D_2|\le x }n(D_2)=&{}\sum_{|D_1|\le x}\sum_{|D_2|\le \frac{x}{|D_1|}}n(D_2) \\
  \nonumber \sim&{}\Theta \frac{ x}{\ze(2)}\sum_{|D_1|\le x}\frac{1}{|D_1|} \\
  \sim &{} \Theta\frac{x\log{x}}{\ze(2)^2}, \label{eq: S1}
  \end{align}
  where the final estimate follows from \eqref{eq: D le x  1/D}. Now, by Lemma \ref{lem: ProportionDiscriminants}, the proportion of fundamental discriminants such that $p | D$ is $\mf{1}{p+1}\cdot$ Hence,
  \begin{align*}
  \sum_{\substack{|D_1D_2|\le x \\ \eta|D_2}}n(D_1)={}&\sum_{|D_1|\le x}n(D_1)\sum_{\substack{|D_2|\le\frac{x}{|D_1|} \ \\ n(D_1)|D_2}}1 \\ ={}&\sum_{|D_1|\le x}\f{n(D_1)}{n(D_1)+1}\sum_{|D_2|\le\frac{x}{|D_1|}}1 \\ \sim{}&\frac{x}{\ze(2)}\sum_{|D_1|\le x}\frac{n(D_1)}{|D_1|(n(D_1)+1)}\cdot
  \end{align*}
  To find an asymptotic for the last sum we again use partial summation. Let
  \[
  B(x)\defeq \sum_{|D_1|\le x}\frac{n(D_1)}{n(D_1)+1}\cdot
  \]
  Then, by (i) of Theorem \ref{th: Pollack},
  \begin{align*}
    \sum_{\substack{|D_1|\le x\\ n(D_1)\le (\log{x})^{\f13}}}\frac{n(D_1)}{n(D_1)+1}={} & \sum_{\substack{k=1 \\p_k\le (\log{x})^{\f13}}}^{\infty}\frac{p_k}{p_k+1}\# \{|D_1|\le x \; : \; n(D_1)=p_k\} \\
  \sim{} & \al\frac{x}{\ze(2)}
  \end{align*}
  where
  \[
  \al= \sum_{k=1}^{\infty}\frac{p_k^2}{2(p_k+1)^2}\prod_{j=1}^{k-1}\frac{p_j+2}{2(p_j+1)}\cdot
\]
  Now, by (ii) of Theorem \ref{th: Pollack},
  \[
  \sum_{\substack{|D_1|\le x\\ n(D_1)\gt (\log{x})^{\f13}}}\frac{n(D_1)}{n(D_1)+1}\le \sum_{\substack{|D_1|\le x\\ n(D_1)\gt (\log{x})^{\f13}}}n(D_1)=o(x).
  \]
  Hence,
  \begin{equation*}\label{eq: B(x)}
   B(x)= \sum_{\substack{|D_1|\le x\\ n(D_1)\le (\log{x})^{\f13}}}\frac{n(D_1)}{n(D_1)+1} + \sum_{\substack{|D_1|\le x\\ n(D_1)\gt (\log{x})^{\f13}}}\frac{n(D_1)}{n(D_1)+1}
    \sim\al\frac{x}{\ze(2)}\cdot
  \end{equation*}

Therefore,
\[
\sum_{|D_1|\le x}\frac{n(D_1)}{|D_1|(n(D_1)+1)}=B(x)f(x)-B(1)f(1)-\int_{1}^{x}B(t)f'(t)\dt\sim \al\frac{\log{x}}{\ze(2)},
\]
so that
\begin{equation}\label{eq: S2}
  \sum_{\substack{|D_1D_2|\le x \\ \eta|D_2}}n(D_1)\sim \al\frac{x\log{x}}{\ze(2)^2}\cdot
\end{equation}
Finally,
\[
\sum_{\substack{|D_1D_2|\le x \\ \eta|D_2}}n(D_2)=\sum_{|D_2|\le x}n(D_2) \sum_{|D_1|\le\f{x}{|D_2|}}\frac{1}{n(D_1)+1}\cdot
\]
To get an estimate for the inner sum, let
  \[
  C(x)\defeq \sum_{|D_1|\le \f{x}{|D_2|}}\frac{1}{n(D_1)+1}\cdot
  \]
  Then, by (i) of Theorem \ref{th: Pollack},
  \begin{align*}
    \sum_{\substack{|D_1|\le \f{x}{|D_2|}\\ n(D_1)\le (\log{x})^{\f13}}}\frac{1}{n(D_1)+1}={} & \sum_{\substack{k=1 \\p_k\le (\log{x})^{\f13}}}^{\infty}\frac{1}{p_k+1}\# \{|D_1|\le \f{x}{|D_2|} \; : \; n(D_1)=p_k\} \\
  \sim{} & \be\frac{x}{\ze(2)|D_2|}
  \end{align*}
  where
  \[
  \be= \sum_{k=1}^{\infty}\frac{p_k}{2(p_k+1)^2}\prod_{j=1}^{k-1}\frac{p_j+2}{2(p_j+1)}\cdot
\]
 On the other hand,
  \begin{align*}
    \sum_{\substack{|D_1|\le \f{x}{|D_2|}\\ n(D_1)\gt (\log{x})^{\f13}}}\frac{1}{n(D_1)+1}\le{} & \sum_{|D_1|\le \frac{x}{|D_2|}}\frac{1}{(\log{x})^{\f13}+1} \\
    \sim{} & \frac{x}{|D_2|\ze(2)((\log{x})^{\f13}+1)} \\
    \le{} & \frac{x}{(\log{x})^{\f13}+1} \\
    ={} & o(x).
  \end{align*}
  Hence,
  \begin{equation*}\label{eq: B(x)}
   C(x)= \sum_{\substack{|D_1|\le \f{x}{|D_2|}\\ n(D_1)\le (\log{x})^{\f13}}}\frac{1}{n(D_1)+1} + \sum_{\substack{|D_1|\le \f{x}{|D_2|}\\ n(D_1)\gt (\log{x})^{\f13}}}\frac{1}{n(D_1)+1}
    \sim\be\frac{x}{\ze(2)|D_2|}\cdot
  \end{equation*}
From this we see that
\begin{equation}\label{eq: S_3}
 \sum_{\substack{|D_1D_2|\le x \\ \eta|D_2}}n(D_2)\sim \sum_{|D_2|\le x}\be\frac{n(D_2)x}{\ze(2)|D_2|}\sim \Theta\be\frac{x\log{x}}{\ze(2)^2},
\end{equation}
where the last estimate follows from partial summation and applying Theorem \ref{th: Pollack}. Therefore, plugging \eqref{eq: S1}, \eqref{eq: S2} and \eqref{eq: S_3} into \eqref{eq: Numerator} shows that
\[
 \sum_{|D_1D_2|\le x}\eta\sim (\Theta+\al-\Theta\be)\frac{x\log{x}}{\ze(2)^2}\cdot
\]
This together with \eqref{eq: denominator} completes the proof.
\end{proof}
\begin{remark}
We can give the following explanation of why Linowitz and Thompson's Conjecture \ref{con: ConLola} was slightly off from the correct number: the result from Theorem \ref{th: LolaMainTheorem} is not uniform in $k$ for the choice of the $p_k$ (we fix a set of primes beforehand), while the result from Theorem \ref{th: Pollack} is uniform in $k$ satisfying $p_k\le (\log{x})^{\f13}$. In order to make Linowitz and Thompson's heuristic argument rigorous we would first need to show that Theorem \ref{th: LolaMainTheorem} holds uniformly in $k$ such that $p_k\le f(x)$ for some function $f$ with $f(x) \xrightarrow[x\to \infty]{} \infty$. Then,
\begin{align*}
  \frac{\sum_{|D_1D_2|\le x}\eta(D_1,D_2)}{\sum_{|D_1D_2|\le x} 1} =\sum_{p_k\le f(x)}{}&p_k\, \mbox{Prob}(\eta(D_1,D_2)=p_k) \\  {}& + \sum_{p_k\gt f(x)}p_k\, \mbox{Prob}(\eta(D_1,D_2)=p_k) \\
  \xrightarrow[x\to \infty]{}{}&\theta + \mu,
\end{align*}
where $\theta$ is the conjectured constant \eqref{eq: theta} and
\[
\mu=\lim_{x\to\infty} \sum_{p_k\gt f(x)}p_k\, \mbox{Prob}(\eta(D_1,D_2)=p_k).
\]
Linowitz and Thompson conjectured that $\mu=0$, but according to Theorem \ref{th: Main Theorem}, $\mu$ does make a small contribution.
\end{remark}
\textbf{Acknowledgments} I would like to thank my PhD supervisor Lola Thompson for guiding me throughout this work and taking the time to give me suggestions about the paper. I would also like to thank the anonymous referee for reading the paper carefully and offering helpful remarks.

\printbibliography

\end{document}